\documentclass[a4paper,11pt]{amsart}
\usepackage{latexsym,amssymb}

\DeclareMathOperator{\RMlt}{RMlt}

\DeclareMathOperator{\Aut}{Aut}

\newtheorem{theorem}{Theorem}[section]
\newtheorem{proposition}[theorem]{Proposition}
\newtheorem{lemma}[theorem]{Lemma}

\newtheorem{definition}[theorem]{Definition}

\newcounter{claim}
\renewcommand\theclaim{(\arabic{section}.\arabic{claim})}
\newenvironment{claim}{\refstepcounter{claim}\par\bigskip\noindent\textbf{\theclaim}}{\par\smallskip}

 \DeclareMathOperator{\rmlt}{RMlt}  

\DeclareMathOperator{\GL}{GL}

\title{Linear groups as right multiplication groups of quasifields}
\author{G\'abor P. Nagy} 
\email{nagyg@math.u-szeged.hu}
\address{Bolyai Institute, University of Szeged, Aradi v\'ertan\'uk tere 1, H-6720 Sze\-ged (Hungary)}
\thanks{The author is a Janos Bolyai Research Fellow. Supported by TAMOP-4.2.2/B-10/1-2010-0012 project of Hungary.}

\begin{document}

\begin{abstract}
For quasifields, the concept of parastrophy is slightly weaker than isotopy. Parastrophic quasifields yield isomorphic translation planes but not conversely. We investigate the right multiplication groups of finite quasifields. We classify all quasifields having an exceptional finite transitive linear group as right multiplication group. The classification is up to parastrophy, which turns out to be the same as up to the isomorphism of the corresponding translation planes. 
\end{abstract}

\maketitle

\section{Introduction}

A \textit{translation plane} is often represented by an algebraic structure called a \textit{quasifield.} Many properties of the translation plane can be most easily understood by looking at the appropriate quasifield. However, isomorphic translation planes can be represented by nonisomorphic quasifields. Furthermore, the collineations do not always have a nice representation in terms of operations in the quasifields. 

Let $p$ be a prime number and $(Q,+,\cdot)$ be a quasifield of finite order $p^n$. We identify $(Q,+)$ with the vector group $(\mathbb{F}_p^n, +)$. With respect to the multiplication, the set $Q^*$ of nonzero elements of $Q$ form a \textit{loop.} The \textit{right multiplication maps} of $Q$ are the maps $R_a:Q\to Q$, $xR_a=x\cdot a$, where $a,x\in Q$. By the right distributive law, $R_a$ is a linear map of $Q=\mathbb{F}_p^n$. Clearly, $R_0$ is the zero map. If $a\neq 0$ then $R_a\in \GL(n,p)$. In geometric context, the set of right translations are also called the \textit{slope set} or the \textit{spread set} of the quasifield $Q$, cf. \cite[Chapter 5]{Handbook}. 

The \textit{right multiplication group} $\rmlt(Q)$ of the quasifield $Q$ is the linear group generated by the nonzero right multiplication maps. It is immediate to see that $\rmlt(Q)$ is a transitive linear group, that is, it acts transitively on the set of nonzero vectors of $Q=\mathbb{F}_p^n$. The complete classification of finite transitive linear groups is known, the proof relies on the classification theorem of finite simple groups. Roughly speaking, there are four infinite classes and $27$ exceptional constructions of finite transitive linear groups. 

The first question this paper deals with asks which finite transitive linear group can occur as right multiplication group of a finite quasifield. It turns out that some of the exceptional finite transitive linear groups may happen to be the right multiplication group of a quasifield. Since these groups are relatively small, in the second part of the paper, we are able to give an explicit classification of all quasifields whose right multiplication group is an exceptional finite linear group. These results are obtained with computer calculations using the computer algebra system GAP4 \cite{GAP4} and the program CLIQUER \cite{Cliquer}. 

Right multiplication groups of quasifields have not been studied intensively. The most important paper in this field is \cite{Kallaher} by M. Kallaher, containing results about finite quasifields with solvable right multiplication groups. Another related paper is \cite{NagyG} which deals with the group generated by the left and right multiplication maps of finite semifields. 

Finally, we notice that our results can be interpreted in the language of the theory of finite loops, as well. \textit{Loops} arise naturally in geometry when coordinatizing point-line incidence structures. Most importantly, the multiplicative structure $(Q^*,\cdot)$ of $Q$ is a loop. In fact, any finite loop $\hat Q$ gives rise to a quasifield, provided the right multiplication maps of $\hat{Q}$ generate a group which is permutation isomorphic to a subgroup of $\GL(n,q)$, where the latter is considered as a permutation group on the nonzero vectors of $\mathbb{F}_q^n$. Therefore in this paper, we investigate loops whose right multiplication groups are contained is some finite linear group $\GL(n,q)$. 

\medskip

There are many excellent surveys and monographs on translation planes and quasifields, see \cite{HughesPiper, Handbook, Luneburg} and the references therein. Our computational methods have similarities with those in \cite{CharnesDempwolff,Dempwolff}.

\section{Translation planes, spreads and quasifields}
\label{sec:geom}

For the shake of completeness, in this section we repeat the definitions of concepts which are standard in the theory of translation planes. Moreover, we briefly explain the relations between the automorphisms of these mathematical objects. 

\medskip

Let $\Pi$ be a finite projective plane. The line $\ell$ of $\Pi$ is a translation line if the translation group with respect to $\ell$ acts transitively (hence regularly) on the set of points $\Pi\setminus \ell$. Assume $\ell$ to be a translation line and let $\Pi^\ell$ be the affine plane obtained from $\Pi$ with $\ell$ as the line at infinity. Then $\Pi^\ell$ is a translation plane. By the theorems of Skornyakov-San Soucie and Artin-Zorn \cite[Theorem 6.18 and 6.20]{HughesPiper}, $\Pi$ is either Desarguesian or contains at most one translation line. This means that two finite translation planes are isomorphic if and only if the corresponding projective planes are. 

The relation between translation planes and quasifields is usually explained using the notion of \textit{(vector space) spreads.} 

\begin{definition}
Let $V$ be a vector space over the field $F$. We say that the collection $\sigma$ of subspaces is a \emph{spread} if (1) $A,B \in \sigma$, $A\neq B$ then $V=A\oplus B$, and (2) every nonzero vector $x\in V$ lies in a unique member of $\sigma$. The members of $\sigma$ are the \emph{components} of the spread. 
\end{definition}

If $\sigma$ is a spread in $V$ then Andr\'e's construction yields a translation plane $\Pi(\sigma)$ by setting $V$ as the set of points and the translates of the components of $\sigma$ as the set of lines of the affine plane. Conversely, if $\Pi$ is a finite translation plane with origin $O$ then we identify the point set of $\Pi$ with the group $\mathcal{T}(\Pi)$ of translations. As $\mathcal{T}(\Pi)$ is an elementary Abelian $p$-group, $\Pi$ becomes a vector space over (some extension of) $\mathbb{F}_p$ and the lines through $O$ are subspaces forming a spread $\sigma(\Pi)$. 

Andr\'e's construction implies a natural identification of the components of the spread with the parallel classes of the affine lines, and the points at infinity of the corresponding affine plane. 

The approach by spreads has many advantages. For us, the most important one is that they allow explicit computations in the group of collineation of the translation plane. Let $\Pi$ be a nondesarguesian translation plane and let us denote by $\mathcal{T}(\Pi)$ the group of translations of $\Pi$. The full group $\Aut(\Pi)$ of collineations contains $\mathcal{T}(\Pi)$ as a normal subgroup. Up to isomorphy, we can choose a unique point of origin $O$ in $\Pi$. The stabilizer $\mathcal{C}_O(\Pi)$ of $O$ in $\Aut(\Pi)$ is the \textit{translation complement} of $\Pi$ with respect to $O$. The full group $\Aut(\Pi)$ of collineations is the semidirect product of $\mathcal{T}(\Pi)$ and $\mathcal{C}_O(\Pi)$. In particular, $\mathcal{C}_O(\Pi)$ has the structure of a linear group of $V$. By \cite[Theorem 2.27]{Handbook}, the collineation group $\mathcal{C}_O(\Pi)$ is essentially the same as the group of automorphisms of the associated spread, where the latter is defined as follows.

\begin{definition}
Let $\sigma$ be a spread in the vector space $V$. The automorphism group $\Aut(\sigma)$ consists of the additive mappings of $V$ that permutes the components of $\sigma$ among themselves. 
\end{definition}

As the translations act trivially on the infinite line, the permutation action of $\Aut(\sigma)$ is equivalent with the action of $\Aut(\Pi)$ on the line at infinity.

Spreads are usually represented by a spread set of matrices. Fix the components $A,B$ of the spread $\sigma$ with underlying vector space $V$. The direct sum decomposition $V=A\oplus B$ defines the projections $p_A:V\to A$, $p_B:V\to B$. As for any component $C\in \sigma\setminus\{A,B\}$ we have $A\cap C=B\cap C=0$, the restrictions of $p_A$ and $p_B$ to $C$ are bijections $C\to A$, $C\to B$. Therefore, the map $u_C:A\to B$, $xu_C=(xp_A^{-1})p_B$ is a linear bijection from $A$ to $B$. When identifying $A,B$ with $\mathbb{F}_p^k$, $u_C$ can be given in matrix form $U_C$. The set $\mathcal{S}(\sigma)=\{U_C \mid C \in \sigma\}$ is called the \textit{spread set of matrices representing $\sigma$ relative to axes $(A,B)$.} This representation depends on the choice of $A,B\in \sigma$. It is also possible to think at a spread set as the collection $\mathcal{S}'(\sigma)=\{u_Du_C^{-1} \mid D \in \sigma \}$ of linear maps $A\to A$, with fixed $C$. 

A spread set $\mathcal{S}\subseteq \hom(A,B)$ can be characterized by the following property: For any elements $x\in A \setminus \{0\}$, $y\in B\setminus \{0\}$, there is a unique map $u\in \mathcal{S}$ such that $xu=y$. Indeed, if $\mathcal{S}=\mathcal{S}(\sigma)$ then $u=u_C$ where $C$ is the unique component containing $x\oplus y$. Conversely, $\mathcal{S}$ defines the spread
\[ \sigma(\mathcal{S})=\{0\oplus B, A \oplus 0\} \cup \{ \{ x\oplus xu \mid x \in A \} \mid u \in \mathcal{S} \} \]
with underlying vector space $V=A \oplus B$. 

\begin{definition}
The autotopism group of the spread set $\mathcal{S}$ of $k\times k$ matrices over $F$ consists of the pairs $(T,U)\in \GL(k,F) \times \GL(k,F)$ such that $T^{-1}\mathcal{S}U=\mathcal{S}$. 
\end{definition}

\cite[Theorems 5.10]{Handbook} says that autotopisms of spread sets relative to axes $(A,B)$ and automorphism of spreads fixing the components $A,B$ are essentially the same.

\medskip

By fixing a nondegenerate quadrangle $o,e,x,y$, any projective plane can be coordinatized by a \textit{planar ternary ring} (PTR), see \cite{HughesPiper}. Let $\Pi$ be a translation plane and fix affine points $o,e$ and infinite points $x,y$. Then, the coordinate PTR becomes a \textit{quasifield.}

\begin{definition}
The finite set $Q$ endowed with two binary operations $+$, $\cdot$ is called a \emph{finite (right) quasifield,} if
\begin{enumerate}
\item[(Q1)] $(Q,+)$ is an Abelian group with neutral element $0\in Q$, 
\item[(Q2)] $(Q\setminus\{0\},\cdot)$ is a loop,
\item[(Q3)] the right distributive law $(x+y)z=xz+yz$ holds, and,
\item[(Q4)] $x\cdot 0=0$ for each $x\in Q$.
\end{enumerate}
\end{definition}

The link between tranlation planes and quasifields can be extended to spread sets, as well. In fact, the set $\mathcal{S}(Q)=\{R_x \mid x\in Q\}$ of nonzero right multiplication maps of $Q$ is a spread set relative to the infinite points of the $x$- and $y$-axes of the coordinate system. Collineations correspond to \textit{autotopisms} of the quasifield.

\begin{definition}
Let $(Q,+,\cdot)$ be a quasifield, $S,T,U:Q\to Q$ bijections such that $S,U$ are additive and $0T=0$. The triple $(S,T,U)$ is said to be an \emph{autotopism} of $Q$ if for all $x,y\in Q$, the identity $xS\cdot yT=(x\cdot y)U$ holds. 
\end{definition}

It is easy to see that the triple $(S,T,U)$ is an autotopism of the quasifield $Q$ if and only if the pair $(S,U)$ is an autotopism of the associated spread set $\mathcal{S}(Q)$. 

We summarize the above considerations in the next proposition.

\begin{proposition} \label{pr:collgrps}
Let $\Pi$ be a translation plane, $\sigma$ the associated spread. Let $A,B$ be fixed components of $\sigma$ and $a,b$ the associated infinite points of $\Pi$.  Let $\mathcal{S}$ be the spread set of $\sigma$ relative to axes $(A,B)$. Let $c,d$ be arbitrary affine points of $\Pi$ such that $a,b,c,d$ are in general position, and let $(Q,+,\cdot)$ be the coordinate quasifield of $\Pi$ with respect to the quadrilateral $abcd$. Then the following groups are isomorphic.
\begin{enumerate}
\item The autotopism group of $\mathcal{S}$.
\item The stabilizer subgroup of the components $A,B$ in $\Aut(\sigma)$.
\item The stabilizer subgroup of the triangle $abc$ in the full collineation group $\Aut(\Pi)$. 
\item The autotopism group of $Q$. 
\end{enumerate}
\end{proposition}

In particular, the structure of the autotopism group of the coordinate quasifield does not depend on the choice of the base points $c,d$.

\section{Isotopy, parastrophy and computation}

When investigating the isomorphism between translation planes, the \textit{isotopy of quasifields} is a central concept. We borrow the concept of \textit{parastrophy} from the theory of loops in order to define a wider class of equivalence for quasifields. 

\begin{definition}
Let $\mathcal{S}, \mathcal{S}'$ be spread sets of matrices in $\GL(d,p)$. We say that $\mathcal{S}, \mathcal{S}'$ are 
\begin{enumerate}
\item \emph{isotopes} if there are matrices $T,U \in\GL(d,p)$ such that $T^{-1}\mathcal{S}U=\mathcal{S}'$ holds. 
\item \emph{parastrophes} if there are matrices $T,U \in\GL(d,p)$ such that $T^{-1}\mathcal{S}U=\mathcal{S}'$ or $T^{-1}\mathcal{S}U=(\mathcal{S}')^{-1}$ holds. 
\end{enumerate}
Analogously, we say that the quasifields $Q,Q'$ are isotopic (parastrophic) if their sets of nonzero right multiplications of matrices are isotopic (parastrophic) as spread sets of matrices. 
\end{definition}

In Section \ref{sec:geom}, we explained the method of obtaining the spread set $\mathcal{S}$ of matrices from the spread $\sigma$ by fixing the components $A,B$. It follows that interchanging $A$ and $B$, the resulting spread set of matrices will be $\mathcal{S}^{-1}=\{u^{-1} \mid u \in \mathcal{S} \}$. Hence, $\mathcal{S}$ and $\mathcal{S}^{-1}$ determine isomorphic translation planes. Taking into account \cite[Propositions 5.36 and 5.37]{Handbook}, we obtain that parastrophic quasifields (or spread sets) determine isomorphic translation planes. On the one hand, the next proposition shows that the right multiplication group of a quasifield is parastrophy invariant. On the other hand, it will give us an effective method for the computation of parastrophy for quasifields with ``small'' right multiplication group. 

\begin{lemma} \label{lm:normal}
Let $\mathcal{S}_1, \mathcal{S}_2$ be spread sets of matrices in $\GL(d,p)$. Assume that $1 \in \mathcal{S}_1, \mathcal{S}_2$ and define the transitive linear groups $G_1=\langle \mathcal{S}_1 \rangle$ and $G_2=\langle \mathcal{S}_2 \rangle$ generated by the spread sets. If $(T,U)$ defines a parastrophy between $\mathcal{S}_1$ and $\mathcal{S}_2$ then $T^{-1}U\in \mathcal{S}_2$ and $T^{-1}G_1T =U^{-1}G_1U =G_2$. In particular, if $G_1=G_2$  then $U,T \in N_{\GL(d,p)}(G_1)$. 
\end{lemma}
\begin{proof}
As $\langle \mathcal{S} \rangle=\langle \mathcal{S}^{-1} \rangle$, it suffices to deal with isotopes. Since $1\in \mathcal{S}_1$, we have $T^{-1}U=T^{-1}1U\in \mathcal{S}_2$. Moreover, $T^{-1}\mathcal{S}_1 T=T^{-1}\mathcal{S}_1 U \cdot U^{-1}T =\mathcal{S}_2\,U^{-1}T \subseteq G_2$, which implies $T^{-1}G_1 T=G_2$. The equation $U^{-1}G_1U=G_2$ follows from $T^{-1}U\in G_2$. 
\end{proof}

We will use the above lemma to compute the classification of quasifields up to parastrophy. 

\begin{proposition} \label{pr:parast}
Let $\mathcal{S}, \mathcal{S}'$ be spread sets of matrices in $\GL(d,p)$. Assume that $1\in \mathcal{S}, \mathcal{S}'$ and $G=\langle \mathcal{S} \rangle = \langle \mathcal{S}' \rangle$.  Let $G^*,G^{**}$ be the permutation groups acting on $G$, where the respective actions are the right regular action of $G$ on itself, and the action of $N_{\GL(d,p)}(G)$ on $G$ by conjugation. Let $\iota$ be the inverse map $g\mapsto g^{-1}$ on $G$. Define the permutation group $G^\sharp=\langle G^*, G^{**},\iota \rangle$ acting on $G$. Then, $\mathcal{S}, \mathcal{S}'$ are parastrophic if and only if they lie in the same $G^\sharp$-orbit. 
\end{proposition}
\begin{proof}
For any $U\in G$, $T\in N_{\GL(d,p)}(G)$, the sets $T^{-1}\mathcal{S}T$, $\mathcal{S}U$ and $\mathcal{S}^{-1}$ are parastrophes of $\mathcal{S}$. Hence, if $\mathcal{S}, \mathcal{S}'$ are in the same $G^\sharp$-orbit then they are parastrophes. Conversely, assume that the pair $(T,U)$ defines an isotopy between $\mathcal{S}$ and $\mathcal{S}'$. By Lemma \ref{lm:normal}, $\mathcal{S}'=T^{-1}\mathcal{S}T\cdot T^{-1}U$ where $T\in N_{\GL(d,p)}(G)$ and $T^{-1}U\in G$, that is, they are in the same $G^\sharp$-orbit. The proof goes similarly for the case of parastrophy.
\end{proof}

In general, the explicit computation of the collineation group of a translation plane is very challenging. Another application of Lemma \ref{lm:normal} is the computation of the autotopism group of a quasifield, that is, the computation of the stabilizer of two infinite points in the full collineation group of the corresponding translation plane. 

\begin{proposition} \label{pr:autotgr}
Let $\mathcal{S}$ be a spread set of matrices in $\GL(d,p)$ with $1\in \mathcal{S}$ and $G=\langle \mathcal{S} \rangle$. Define the group
\[H=\{(T,U) \in N_{\GL(d,p)}(G)^2 \mid T^{-1}U\in G\}\]
and the permutation action $\Phi:G\times H\to G$ of $H$ on $G$ by
\[\Phi:(X,(T,U)) \mapsto T^{-1}XU.\]
Then, $H$ and $\Phi$ are well defined. The isotopism group of $\mathcal{S}$ is the setwise stabilizer of $\mathcal{S}$ in $H$ with respect to the action $\Phi$. 
\end{proposition}
\begin{proof}
In order to see that $H$ is a group, take elements $(T,U), (T_1,U_1)\in H$. On the one hand, $(T^{-1})^{-1}U^{-1}=T(T^{-1}U)^{-1}T^{-1} \in TGT^{-1}=G$, which implies $(T^{-1},U^{-1}) \in H$. On the other hand, $(TT_1,UU_1)\in H$ follows from
\[(TT_1)^{-1}UU_1 = T_1^{-1}(T^{-1}U)T_1 \cdot T_1^{-1}U_1 \in T_1^{-1}GT_1\cdot G =G.\]
Since $T^{-1}XU=T^{-1}XT\cdot T^{-1}U\in G$ holds for all $X \in G$, $H$ and $\Phi$ are well defined. The claim for the autotopism group of $\mathcal{S}$ follows from Lemma \ref{lm:normal}. 
\end{proof}

As for an exceptional finite transitive linear group $G$, $N_{\GL(d,p)}(G)$ is also exceptional, it is a small subgroup of $\GL(d,p)$ and the autotopism group of $\mathcal{S}$ is computable by GAP4. Using Proposition \ref{pr:collgrps}, we obtain a straightforward method for computing the stabilizer of two infinite points of our translation planes. However, the stabilizer of some infinite points is not an invariant of the translation plane in general. 

More precisely, let $\Pi$ be a translation plane with infinite line $\ell_\infty$ and $P=\{a_1,\ldots,a_t\} \subseteq \ell_\infty$ be a set of infinite points. Denote by $\mathcal{B}_P$ the pointwise stabilizer of $P$ in $\Aut(\Pi)$. In the rest of this section we show that under some circumstances, the structure of $\mathcal{B}_P$ is an invariant of $\Pi$. We start with a lemma on general permutation groups. 

\begin{lemma} \label{lm:stab}
Let $G$ be a group acting on the finite set $X$ (not necessarily faithfully). For any $Y\subseteq X$, we denoty by $F(Y)$ the pointwise stabilizer of $Y$. Let $Y_1,Y_2$ be subsets of $X$ such that $|Y_1|=|Y_2|< |X|/2$, and suppose that $F(Y_i)$ acts transitively on $X\setminus Y_i$, $i=1,2$. Then, there is an element $g\in G$ with $Y_2^g=Y_1$. In particular, the subgroups $F(Y_1), F(Y_2)$ are conjugate in $G$. 
\end{lemma}
\begin{proof}
Up to the action of $G$ on the orbit $Y_2^G$, we may assume that $|Y_1\cap Y_2|\geq |Y_1\cap Y_2^g|$ for all $g\in G$. Suppose that $Y_1\neq Y_2$ and take elements $x\in X\setminus(Y_1\cup Y_2)$, $y_1\in Y_1\setminus Y_2$, $y_2\in Y_2\setminus Y_1$. As $F(Y_i)$ acts transitively on $X\setminus Y_i$, there are elements $g_1\in F(Y_1)$, $g_2 \in F(Y_2)$ such that $x^{g_1}=y_2$ and $x^{g_2}=y_1$. Put $h=g_1^{-1}g_2$. Then $h\in F(Y_1\cap Y_2)$ and $y_2^h=y_1\in Y_1$. This implies $|Y_1\cap Y_2^h|>|Y_1\cap Y_2|$, a contradiction. 
\end{proof}

We can apply the lemma for the stabilizer of infinite points of a translation plane. 

\begin{proposition} \label{pr:infstab}
Let $\Pi$ be a translation plane of order $q$, with infinite line $\ell_\infty$. For a subset $P\subseteq \ell_\infty$, let $\mathcal{B}_P$ denote the pointwise stabilizer subgroup in $\Aut(\Pi)$. Fix the integer $t\leq (q+1)/2$ and define the set
\[ \mathcal{D}_t=\{P\subseteq \ell_\infty \mid |P|=t \mbox{ and $\mathcal{B}_P$ act transitively on $\ell_\infty \setminus P$} \}. \]
Then, $\Aut(\Pi)$ acts transitively on $\mathcal{D}_t$. In particular, the structure of $\mathcal{B}_P$ (as a permutation group on $\Pi\cup \ell_\infty$) does not depend on the particular choice $P\in \mathcal{D}_t$. \qed
\end{proposition}

\section{Sharply transitive sets and permutation graphs}

Let $G$ be a permutation group acting on the finite set $\Omega$, $n=|\Omega|$ is the degree of $G$. The subset $S\subseteq G$ is a \textit{sharply transitive set of permutations} if for any $x,y\in \Omega$ there is a unique element $\sigma\in S$ such that $x\sigma=y$. Sharply transitive sets can be characterized by the property $|S|=|\Omega|$ and, for all $\sigma,\tau \in S$, $\sigma \tau^{-1}$ is fixed point free. In particular, if $1\in S$ then all $\sigma \in S\setminus \{1\}$ are fixed point free elements of $G$. 

\begin{definition}
Let $G$ be a finite permutation group. The pair $\mathcal{G}=(V,\mathcal{E})$ is the \emph{permutation graph} of $G$, where $V$ is the set of fixed point free elements of $G$ and the edge set $\mathcal{E}$ consists of the pairs $(x,y)$ where $xy^{-1}\in V$. 
\end{definition}

The set $K$ of vertices is a \textit{$k$-clique} if $|K|=k$ and all elements of $K$ are connected. Sharply transitive subsets of $G$ containing $1$ correspond precisely the $(n-1)$-cliques of the permutation graph of $G$. Assume that the action of $G$ on $\Omega$ is imprimitive and let $\Omega'$ be a nontrivial block of imprimitivity. Let $S$ be a sharply transitive subset of $G$ and define the subset $S'=\{\sigma \in S \mid \Omega' \sigma =  \Omega'\}$ of $S$. Then, the restriction of $S'$ to $\Omega'$ is a sharply transitive set on $\Omega'$. We will call $\Sigma'$ the \textit{subclique corresponding to the block $\Omega'$}.

For the connection between quasifields and sharply transitive sets of matrices see \cite[Chapter 8]{Handbook}. In our computations, we represent quasifields by the corresponding sharply transitive set of matrices, or, more precisely by the corresponding (maximal) clique of the permutation graph.

\section{Finite transitive linear groups} 
\label{sec:trlingr}

In this section, we give an overview on finite transitive linear groups. For more details and references see \cite[XII.7.5]{HuppertBlackburn} or \cite[Theorem 69.7]{Handbook}. 

Let $p$ be a prime, $V=\mathbb F_p^d$, and $\Gamma=\GL(d,p)$. Let $G\leq \Gamma$ be a subgroup acting transitively on $V^*=V\setminus \{0\}$. Then $G_0\trianglelefteq G \leq N_\Gamma(G_0)$, where we have one of the following possibilities for $G$ and $G_0$:

\begin{enumerate}
\item[1.a)] $G\leq \mathrm{\Gamma L}(1,p^d)$. In particular, $G$ is solvable.
\item[1.b)] $G_0\cong \mathrm{SL}(d/e,p^e)$ with $2\leq e\mid d$. 
\item[2)] $G_0\cong \mathrm{Sp}(d/e,p^e)$ with $e\mid d$, $d/e$ even. 
\item[3)] $p=2$, $d=6e>6$, and $G_0$ is isomorphic to the Chevalley group $G_2(2^e)$. If $d=6$ then $G_0$ is isomorphic to $G_2(2)'$. Notice that the Chevalley group $G_2(2)$ is not simple, its commutator subgroup has index two and the isomorphism $G_2(2)'\cong \mathrm{PSU}(3,3)$ holds. 
\item[4)] There are $27$ exceptional finite transitive linear groups, their structure is listed in Table \ref{tab:ecases}. The information given in the last column can be used to generate the sporadic examples in the computer algebra systems GAP4 \cite{GAP4} in the following way. The split extension of $G$ by the vector group $\mathbb F_p^d$ is $2$-transitive, hence primitive, and can be loaded from the library of primitive groups using the command \texttt{PrimitiveGroup($p^d$,$k$)}. As $59^2>2500$, case (4.i) is not included in this library, but since this group is regular on $V^*$, it will not be interesting from out point of view. 

We have seven sporadic transitive linear groups which are regular, these are denoted by an asterix. These groups have been found by Dickson and Zassenhaus, and they are also known as the right multiplication groups of the \textit{Zassenhaus nearfields.} 
\end{enumerate}

\begin{table}
\begin{tabular}{|l|l|l|l|l|}
\hline
Case & Cond. on $p$ & Cond. on $d$ & $G_0$ & Primit. id. of $p^d:G$\\
\hline
(4.a) & $p=5$ & $d=2$ & $\mathrm{SL}(2,3)$ & $15^*,18,19$\\
(4.b) & $p=7$ & $d=2$ & $\mathrm{SL}(2,3)$ & $25^*,29$\\
(4.c) & $p=11$ & $d=2$ & $\mathrm{SL}(2,3)$ & $39^*,42$\\
(4.d) & $p=23$ & $d=2$ & $\mathrm{SL}(2,3)$ &  $59^*$\\
(4.e) & $p=3$ & $d=4$ & $\mathrm{SL}(2,5)$ & $124,126,127,128$\\
(4.f) & $p=11$ & $d=2$ & $\mathrm{SL}(2,5)$ & $56^*,57$\\
(4.g) & $p=19$ & $d=2$ & $\mathrm{SL}(2,5)$ & $86$\\
(4.h) & $p=29$ & $d=2$ & $\mathrm{SL}(2,5)$ & $106^*,110$\\
(4.i) & $p=59$ & $d=2$ & $\mathrm{SL}(2,5)$ & no id; regular\\
(4.j) & $p=3$ & $d=4$ & $2^{1+4}$ & $71,90,99,129,130$\\
(4.k) & $p=2$ & $d=4$ & $A_6$ & $16,17$\\
(4.l) & $p=2$ & $d=4$ & $A_7$ & $20$\\
(4.m) & $p=3$ & $d=6$ & $\mathrm{SL}(2,13)$ & $396$\\
\hline
\end{tabular}
\caption{Exceptional finite transitive linear groups} \label{tab:ecases}
\end{table}

In this paper, without mentioning explicitly, we consider all finite linear groups as a permutation group acting on the nonzero vectors of the corresponding linear space.

\section{Non-existence results for finite right quasifields} 
\label{sec:nonexistence}

In this section, let $(Q,+,\cdot)$ be a finite right quasifield of order $p^d$ with prime $p$. Write $G=\RMlt(Q^*)$ for the right multiplication group of $Q$. As in Section \ref{sec:trlingr}, we denote by $G_0$ a characteristic subgroup of $G$. Moreover, we denote by $S$ the set of nontrivial right multiplication maps of $Q$. Then $1\in S$ and $S$ is a sharply transitive set of permutations in $G$. We write  $\mathcal{G}=(V,\mathcal{E})$ for the permutation graph of $G$. Remember that $|S|=p^d-1$ and $S\setminus\{1\}$ is a clique of size $p^d-2$ in $\mathcal{G}$. 

\begin{proposition} \label{prop:infcl}
Assume that $G$ is a transitive linear group belonging to the infinite classes (1)-(3). Then one of the following holds:
\begin{enumerate}
\item $G\leq \mathrm{\Gamma{}L}(1,p^d)$. 
\item $G \triangleright \mathrm{SL}(d/e,p^e)$ for some divisor $e<d$ of $d$.
\item $p$ is odd and $G \triangleright \mathrm{Sp}(d/e,p^e)$ for some divisor $e$ of $d$.
\end{enumerate}
\end{proposition}
\begin{proof}
We have to show that if $p=2$ then the cases $G_0\cong \mathrm{Sp}(d/e,p^e)$ and $G_0\cong G_2(2^e)'$ are not possible. Recall that the transitive linear group $G_2(2^e)$ is a subgroup of $\mathrm{Sp}(6,2^e)$. The impossibility of both cases follow from \cite[Theorem 1]{MuellerNagy}. 
\end{proof}

\begin{proposition} \label{prop:except}
If $G$ is an exceptional finite transitive linear group, then it is either a regular linear group or one of those in Table \ref{tab:easrmlt}.
\end{proposition}

\begin{table}
\begin{tabular}{|l|l|l|l|l|}
\hline
Case & $p^d$ & $G_0$ & Orders (primitive id.)\\
\hline
(4.a) & $5^2$ & $\mathrm{SL}(2,3)$ & 48 (18), \textit{96 (19)}\\
(4.b) & $7^2$ & $\mathrm{SL}(2,3)$ & 144 (29)\\
(4.c) & $11^2$ & $\mathrm{SL}(2,3)$ & 240 (42)\\
(4.e) & $3^4$ & $\mathrm{SL}(2,5)$ & 960 (128)\\
(4.f) & $11^2$ & $\mathrm{SL}(2,5)$ & \textit{600 (57)}\\
(4.g) & $19^2$ & $\mathrm{SL}(2,5)$ & 1080 (86)\\
(4.h) & $29^2$ & $\mathrm{SL}(2,5)$ & 1680 (110)\\
(4.l) & $2^4$ & $A_7$ & 2520 (20)\\
\hline
\end{tabular}
\caption{Exceptional transitive linear groups as right multiplication groups} \label{tab:easrmlt}
\end{table}

This proposition is proved in several steps.

\begin{claim} \label{claim:no_4k4m}
$G$ cannot be an exceptional transitive linear group of type (4.k) and (4.m). 
\end{claim}
\begin{proof}[Computer proof]
Assume $G$ of type (4.k) and $G_0\cong A_6$. Let $S$ be a Sylow $5$-group of $G$. Let $A,B,C$ the orbits of $S$ such that $B\cup C$ is an orbit of $N_G(S)$. Then, the sets $A,B$ of size $5$ satisfy $|A\cap B^g|\in\{0,2\}$ for all $g\in G$. By \cite[Lemma 2]{MuellerNagy}, $G$ does not contain a sharply transitive set of permutations. 

Assume $G$ of type (4.m). Let $S$ be a Sylow $7$-group of $G$ and $H=N_G(S)$. Then $|H|=28$ and $H$ has orbits of length $28$. One can find $H$-orbits $A,B$ of size $28$ such that $|A\cap B^g|\in\{0,6\}$ for all $g\in G$. Again by \cite[Lemma 2]{MuellerNagy}, $G$ does not contain a sharply transitive set of permutations. 
\end{proof}

\begin{claim} \label{claim:no_4e}
$G$ cannot be an exceptional transitive linear group of type (4.e) and order $240$ or $480$. 
\end{claim}
\begin{proof}[Computer proof]
We proceed similarly to the cases in \ref{claim:no_4k4m}. Assume $G$ of type (4.e) and $|G| \in \{240,480\}$. Let $H$ be the normalizer of a Sylow $5$-group in $G$. Then $|H|=40$ and $H$ has orbits $A,B$ of length $40$ such that $|A\cap B^g|\in\{0,24\}$ for all $g\in G$. This proves the claim by \cite[Lemma 2]{MuellerNagy} with $p=3$.
\end{proof}

Beside finite regular linear groups, the remaining exceptional transitive linear groups are those of type (4.j), that is, when $G_0$ is the extraspecial group $E_{32}^+$ of order $2^5$. In this class, there are five groups, three of them are solvable. The cases of the three solvable group of type (4.j) were left unsolved by M. Kallaher \cite{Kallaher}, as well. The computation is indeed very tedious even with today's hardware and software. 

The construction of the groups of type (4.j) is as follows. $\GL(4,3)$ has a unique subgroup $L_0$ of order $2^5$ which is isomorphic to the extraspecial $2$-group $E_{32}^+$. The normalizer $L=N_{\GL(4,3)}(L_0)$ has order $3840$. $L$ has five transitive linear subgroups containing $L_0$; the orders are $160, 320, 640, 1920, 3840$. Clearly, it suffices to prove the nonexistence of $79$-cliques for the permutation graph of $L$ only. 

The action of $L$ is not primitive; it has blocks of imprimitivity of size $2,8$ and $16$. These blocks are unique up to the action of $L$.

\begin{claim} \label{claim:non2gr}
Let $A$ be a block of $L$ of size $16$. Let $K$ be a $15$-clique corresponding to $A$ and assume that the group $\langle K \rangle$ generated by $K$ is not a $2$-group. Then, $K$ cannot be extended to a $79$-clique of $L$. 
\end{claim}
\begin{proof}[Computer proof]
Let $A$ be a block of size $8$. Denote by $H$ the setwise stabilizer of $A$ in $L$; $H$ stabilizes another block $B$ of size $8$ and $A\cup B$ is a block of size $16$.  Let $\mathcal{A}$ be the set of all $7$-cliques corresponding to the block $A$. As $N_L(H)$ operates on the permutation (sub)graph of $H$, it also acts on $\mathcal{A}$; let $\mathcal{A}_0$ be a set of orbit representatives. We use \cite{Grape} to compute $\mathcal{A}_0$; $|\mathcal{A}_0|=98$. Then in all possible ways, we extend the elements of $\mathcal{A}_0$ to $15$-cliques corresponding to the block $A\cup B$, let $\mathcal{B}$ denote the set of extended cliques. Finally, we filter out the $15$-cliques generating a non-$2$-group and show that none of them can be extended to a $79$-clique. 
\end{proof}

We have to deal with subcliques generating a $2$-group. An important special case is the following. 

\begin{claim} \label{claim:15clique}
Up to conjugacy in $L$, there is a unique $15$-clique $K^*$ corresponding to an imprimitivity block of size $16$ such that the setwise stabilizer in $L$ has order $192$. $K^*$ has the further properties:
\begin{enumerate}
\item The subgroup $\langle K^* \rangle$ has order $32$. 
\item $\langle K^* \rangle$ interchanges two blocks of size $16$ and fixes the other three. 
\end{enumerate}
\end{claim}
\begin{proof}[Computer proof]
Let $S$ be a Sylow $2$-subgroup of $L$. Using \cite{Grape}, one can compute all $15$-cliques of the permutation graph of $S$. Up to conjugacy in $S$, there are $17923$ such cliques, only one of them has stabilizer of size $192$. The properties of $K^*$ are obtained by computer calculations.
\end{proof}

\begin{claim} \label{claim:2gr}
Let $A$ be a block of $L$ of size $16$. Let $K$ be a $15$-clique corresponding to $A$ and assume that the group $\langle K \rangle$ generated by $K$ is a $2$-group. Then, $K$ cannot be extended to a $79$-clique of $L$. 
\end{claim}
\begin{proof}[Computer proof]
Let $S$ be a Sylow $2$-subgroup of $L$. $S$ leaves a block of size $16$, say $A$, invariant. We compute the set $\mathcal A$ of $S$-orbit representatives of the $15$-cliques of the permutation graph of $S$. $\mathcal{A}$ contains precisely one element conjugate to $K^*$, in fact, we assume that $K^*\in\mathcal{A}$. 

For all elements $K \in \mathcal{A}\setminus \{K^*\}$, the computer shows within a few seconds that $K$ cannot be extended to a $79$-clique. For $K^*$, the direct computation takes too long, we therefore give a theoretical proof. Let us assume that $D$ is a clique of size $79$ in the permutation graph of $L$. We may assume that all subcliques of $D$ corresponding to $16$-blocks are conjugate of $K^*$.

Denote by $D_A$ the subclique of $D$ of size $15$, corresponding to $A$. By \ref{claim:15clique}, $\langle D_A\rangle$ interchanges two $16$-blocks, say $B,B'$, and leaves the others invariant. Denote by $D_B$ the subclique of $D$, corresponding  to $B$. Clearly, $D_A\neq D_B$. As $\langle D_B \rangle$ leaves three blocks invariant, there is a $16$-block $C$ which is invariant under all elements of $D_A \cup D_B$. However, as $|C|=16$, $D$ cannot have more than $15$ elements mapping $C$ to $C$, a contradiction to $|D_A\cup D_B|>15$. 
\end{proof}

The claims \ref{claim:non2gr}, \ref{claim:15clique} and \ref{claim:2gr} imply the following result. 

\begin{claim} \label{claim:noextraspec}
$G$ cannot be an exceptional transitive linear group of type (4.j). 
\end{claim}

The combination of the claims \ref{claim:no_4k4m}, \ref{claim:no_4e} and \ref{claim:noextraspec} yields the proof of the Proposition \ref{prop:except}.

\section{Exhaustive search for cliques and their invariants}
\label{sec:exhaustive}

Let $G\leq \GL(d,p)$ be a transitive linear group. We use the program CLIQUER \cite{Cliquer} to compute all cliques of size $p^d-1$ in the permutation graph $\mathcal{G}$ of $G$. For the exceptional transitive linear groups of Table \ref{tab:ecases}, the result is presented in the second column of Table \ref{tab:maxcls}. Proposition \ref{pr:parast} allows us to reduce our results on cliques in exceptional transitive linear groups modulo parastrophy, as shown in the third column of Table \ref{tab:maxcls}. In column 4, we filtered out those cliques which do not generate the whole group $G$. 

\begin{table}
\begin{tabular}{|l||c|c|c|c|}
\hline
type & \# of cliques & up to parastrophy & proper $G$ & \# CCFPs'\\
\hline\hline
(4.a) & $4;8$ & $2;3$ & $1;0$ & $1;0$ \\ \hline
(4.b) & 12 & 4 & 2 & 2 \\ \hline
(4.c) & 16 & 4 & 3 & 3 \\ \hline
(4.e) & 27648 & 32 & 21 & 20 \\ \hline
(4.f) & 6 & 2 & 0 & 0 \\ \hline
(4.g) & 9 & 3 & 3 & 3 \\ \hline
(4.h) & 64 & 9 & 8 & 8 \\ \hline
(4.l) & 450 & 2 & 2 & 1 \\
\hline
\end{tabular}
\caption{Maximal cliques in exceptional transitive linear groups} \label{tab:maxcls}
\end{table}

In the final step, we compute the \textit{Conway-Charnes fingerprint} of all spread sets of matrices, cf. \cite{CharnesDempwolff,MathonRoyle}. The Conway-Charnes fingerprint is an invariant of the translation plane which can be easily computed from any spread set of matrices. 

\begin{definition}
Given a spread set $\mathcal{S}=\{U_1,\ldots,U_{q-1}\}$ of nonzero matrices, one forms a $(q-1) \times (q-1)$ matrix whose $(i,j)$ entry is $0$, $+1$ or $-1$ according as $\det(U_i-U_j)$ is zero, square or nonsquare, respectively. Border this matrix with a leading row and column of 1s (except for the first entry on the diagonal which remains zero) to form a symmetric $q\times q$ matrix $A$ with $0$ on the diagonal, and $\pm 1$ in every non-diagonal entry. Finally, form the matrix $F = AA^t$ (with the product being taken in the rational numbers). The \emph{fingerprint} of the spread set is the multiset of the absolute values of the entries of $F$.
\end{definition}


The last column of Table \ref{tab:maxcls} contains the number of different Conway-Charnes fingerprints of the spread sets of matrices in the corresponding exceptional transitive linear group $G$. One sees that only for two pairs of spread sets do we obtain the same fingerprint.



Let us denote by $Q_1,Q_1'$ the quasifields of order $3^4$ and by $Q_2,Q_2'$ those of order $2^4$. We compute the corresponding autotopism groups $\mathcal{A}_1, \mathcal{A}_1'$ and $\mathcal{A}_2,\mathcal{A}_2'$. GAP4 shows that $\mathcal{A}_1$ and $\mathcal{A}_1'$ are nonisomorphic groups of order $640$, acting transitively on the $80$ points of $\ell_\infty \setminus \{(0), (\infty)\}$. By Proposition \ref{pr:infstab}, $Q_1,Q_1'$ determine nonisomorphic translation planes. $\mathcal{A}_2$ and $\mathcal{A}_2'$ are both isomorphic to $\mathrm{PSL}(2,7)$. Both groups fix a third point $a,a'$ of the infinite line and act transitively on the remaining $14$ infinite points. Hence, both groups are the stabilizer of a triple of infinite points and Proposition \ref{pr:infstab} applies. As the orbit lengths of $\mathcal{A}_2$ and $\mathcal{A}_2'$ are different, we can conclude that the two translation planes are nonisomorphic.

\section{Right multiplication groups of finite right quasifields}

We compile our results in the following theorem. 

\begin{theorem}
Let $(Q,+,\cdot)$ be a finite right quasifield of order $p^d$ with prime $p$. Then, for $G=\RMlt(Q^*)$, one of the following holds:
\begin{enumerate}
\item $G\leq \mathrm{\Gamma{}L}(1,p^d)$ and the corresponding translation plane is a generalized Andr\'e plane. 
\item $G \triangleright \mathrm{SL}(d/e,p^e)$ for some divisor $e<d$ of $d$ with $e\neq d$.
\item $p$ is odd and $G \triangleright \mathrm{Sp}(d/e,p^e)$ for some divisor $e$ of $d$.
\item $p^d\in \{5^2,7^2,11^2,17^2,23^2,29^2,59^2\}$ and $G$ is one of the seven finite sharply transitive linear groups of Zassenhaus \cite{Zassenhaus}. The corresponding translation planes are called Zassenhaus nearfield planes.
\item $p^d\in \{5^2,7^2,11^2\}$, and $G$ is a solvable exceptional transitive linear group. These quasifields and the corresponding translation planes have been given by M. J. Kallaher \cite{Kallaher}. 
\item $p^d=3^4,19^2$ or $29^2$, and the number of translation planes is $21$, $3$ or $8$, respectively. 
\item $p^d=16$ and $G=A_7$. The corresponding translation planes are the Lorimer-Rahilly and Johnson-Walker planes.
\end{enumerate}
\end{theorem}
\begin{proof}
By Proposition \ref{prop:infcl}, (1), (2) or (3) holds if $G$ belongs to one of the infinite classes of finite transitive linear groups. If $G\leq \mathrm{\Gamma{}L}(1,p^d)$ then the corresponding translation plane is a generalized Andr\'e plane by \cite[Theorem 3.1]{Kallaher}. The cases of the Zassenhaus nearfield planes are in (4). The arguments in Sections \ref{sec:nonexistence} and \ref{sec:exhaustive} imply (6) and that there are two quasifields having $A_7$ as right multiplication group. Moreover, the corresponding translation planes are nonisomorphic.  \cite[Corollary 4.2.1]{JhaKallaher} implies that the two translation planes are the Lorimer-Rahilly and Johnson-Walker planes.
\end{proof}

We close this paper with two remarks.
\begin{enumerate}
\item No quasifield $Q$ is known to the author with $\RMlt(Q^*) \triangleright \mathrm{Sp}(d/e,p^e)$. Even the smallest case of $\mathrm{Sp}(4,3)$ is computationally challenging. 
\item The Lorimer-Rahilly and Johnson-Walker translation planes are known to be polar to each other, that is, one spread set of matrices is obtained by transposing the matrices in the other spread set. (See \cite[29.4.4]{Handbook}.)
\end{enumerate}

The GAP programs used in this paper are avaible on the author's web page:
\begin{center}
\verb+http://www.math.u-szeged.hu/~nagyg/pub/rightmlt.html+
\end{center}

\bibliographystyle{plain}

\end{document}